\documentclass[12pt]{article}
\usepackage{color}
\usepackage{graphicx}
\usepackage{colortbl}
\usepackage{array}
\usepackage{amssymb}
\usepackage{amsmath}
\usepackage{mathrsfs}
\usepackage{dcolumn}
\usepackage{longtable}
\usepackage{hhline}
\usepackage{graphics}
\usepackage{amssymb}
\usepackage{amsmath}
\usepackage{amsthm}
\usepackage{mathrsfs}
\usepackage{amsfonts}
\usepackage{latexsym}
\usepackage{color}
\usepackage{epsf}

\newtheorem{thm}{Theorem}[section]
\newtheorem{defn}[thm]{Definition}

\newtheorem{cor}[thm]{Corollary}

\newcommand{\rstf}[2]{\genfrac{[}{]}{0pt}{0}{#1}{#2}}

\def\S{\mathbf{S}}
\def\v{\mathbf{v}}

\title{\bf On the Unified Generalized $q$-Stirling Matrices and Their Matrix Decompositions}
\author{\large{\bf Amerah M. Dibagulun$^{1}$, Charles B. Montero$^{1}$,}\\\large{\bf and Mahid M. Mangontarum$^{1,2}$}\\
$^{1}$Department of Mathematics\\
Mindanao State University-Main Campus\\
Marawi City 9700\\
Philippines\\
$^{2}$Mamitua Saber Institute of Research and Creation\\
Mindanao State University-Main Campus\\
Marawi City 9700\\
Philippines\\
{\tt amerah.dibagulun@msumain.edu.ph}\\
{\tt charles.montero@msumain.edu.ph}\\
{\tt mangontarum.mahid@msumain.edu.ph}
}

\begin{document}

\maketitle

\begin{abstract}
	In this paper, the unified generalized $q$-Stirling matrices are defined in terms of the $q$-analogues of the unified generalized Stirling numbers. Several matrix relations involving these matrices are derived, and matrix decompositions of the unified generalized $q$-Stirling matrices are established. 
	
	
\end{abstract}

\section{Introduction}

The Stirling numbers of the first and second kinds were first introduced by Stirling \cite{Stir} in 1730. In combinatorics, the Stirling numbers of the first kind count the ways of arranging distinct items into a specified number of disjoint circles, whereas the Stirling numbers of the second kind count the ways of partitioning a set of distinct items into a specified number of nonempty, indistinguishable groups. Alternatively, these numbers can be defined as the coefficients in the expansions of factorials in terms of powers and vice versa. Specifically, for a positive integer $t$ and nonnegative integers $n$ and $k$ with $n\geq k$, the Stirling numbers of the first and second kinds, denoted by $s(n,k)$ and $S(n,k)$, respectively, are defined as the coefficients in the following expansions.
\begin{equation}
(t)_n=\sum_{k=0}^ns(n,k)t^k
\end{equation}  
and
\begin{equation}
t^n=\sum_{k=0}^nS(n,k)(t)_k,
\end{equation}
where
\begin{equation*}
(t)_n=t(t-1)(t-2)\cdots(t-(n-1)),\qquad (t)_0=1,
\end{equation*}
is the falling factorial of $t$ of order $n$.

In 1998, Hsu and Shuie \cite{Hsu} introduced a type of generalization by defining a pair of Stirling numbers as the coefficients of linear transformations between generalized factorials involving three arbitrary parameters. Using the notation 
\begin{equation*}
	\left\{S^1(n,k),S^2(n,k)\right\}\equiv
	\left\{S(n,k;\alpha,\beta,\gamma),S(n,k;\beta,\alpha,-\gamma)\right\},
\end{equation*}
the generalized Stirling numbers are defined by the following relations \cite[Equations (1) and (2)]{Hsu}:
\begin{equation}
	(t|\alpha)_n=\sum_{k=0}^{n}S^1(n,k)(t-\gamma|\beta)_k\label{Hsu1}
\end{equation}
and
\begin{equation}
	(t|\beta)_n=\sum_{k=0}^{n}S^2(n,k)(t+\gamma|\alpha)_k.\label{Hsu2}
\end{equation}	
In these equations, $n$ is a nonnegative integer, $\alpha,\beta,\gamma$ may be real or complex numbers, with $(\alpha,\beta,\gamma)\neq(0,0,0)$, and 
\begin{equation}
	(t|\alpha)_n=t(t-\alpha)(t-2\alpha)\cdots(t-(n-1)\alpha),\qquad (t|\alpha)_0=1,
\end{equation}
denotes the generalized factorial of $t$ with increment $\alpha$. It can also be verified that the classical Stirling numbers of the first and second kinds, the binomial coefficients $\binom{n}{k}$, and the classical Lah numbers $L(n,k)$ can be obtained from $S(n,k;\alpha,\beta,\gamma)$ as follows:
\begin{itemize}
	\item Stirling numbers of the first kind: $S(n,k;1,0,0)=s(n,k)$, 
	\item Stirling numbers of the second kind: $S(n,k;0,1,0)=S(n,k)$,
	\item Binomial coefficient: $S(n,k;0,0,1)=\binom{n}{k}$,
	\item Lah numbers: $S(n,k;-1,1,0)=L(n,k)$,
\end{itemize}
respectively, where $L(n,k)$ is defined by
\begin{equation}
	\left\langle t\right\rangle_n=\sum_{k=0}^{n}L(n,k)(t)_k
\end{equation}
and
\begin{equation*}
	\left\langle t\right\rangle_n=t(t+1)(t+2)\cdots(t+n-1)),\qquad \left\langle t\right\rangle_0=1,
\end{equation*}
is the rising factorial of $t$ of order $n$. 

Perhaps one of the most significant features of the generalized Stirling numbers is their ability to unify the various generalizations and extensions of the Stirling numbers introduced by different authors both before and after their formulation. In particular, the numbers considered by the authors listed below can be obtained from $S(n,k;\alpha,\beta,\gamma)$ by assigning suitable values to the three parameters $\alpha$, $\beta$, and $\gamma$:

\begin{itemize}
	\item Koutras' \cite{Kout} noncentral Stirling numbers:
	$$ S(n,k;1,0,a)=s_a(n,k),\qquad	S(n,k;0,1,-a)=S_a(n,k).$$
	These numbers are also equivalent to Broder's \cite{Bro} $r$-Stirling numbers, with $n$ and $k$ replaced by $n-r$ and $k-r$, respectively, and to Carlitz's \cite{Car2} weighted Stirling numbers.
	
	\item Carlitz's \cite{Car1} degenerate Stirling numbers:
	$$ S(n,k;1,\theta,0)=S_1(n,k\mid\theta),\qquad S(n,k;\theta,1,0)=S_2(n,k\mid\theta).$$
	
	\item Benoumhani's \cite{Ben1,Ben2} Whitney numbers of Dowling lattices:
	$$ S(n,k;m,0,-1)=w_m(n,k),\qquad S(n,k;0,m,1)=W_m(n,k).$$
		
	\item Mez\H{o}'s \cite{Mez1} $r$-Whitney numbers:
	$$ S(n,k;m,0,-r)=w_{m,r}(n,k),\qquad S(n,k;0,m,r)=W_{m,r}(n,k).$$
	These numbers are equivalent to the noncentral Whitney numbers introduced by Mangontarum et al. \cite{Mangontarum5} and to Corcino's \cite{Corcino1} $(r,\beta)$-Stirling numbers.
	
	\item Belbachir and Bousbaa's \cite{Bel} translated Whitney numbers:
	$$ S(n,k;\alpha,0,0)=\widetilde{w}_{(a)}(n,k),\qquad S(n,k;0,\alpha,0)=\widetilde{W}_{(a)}(n,k).$$
	These numbers were subsequently developed extensively by Mangontarum et al. \cite{Mangontarum,Mangontarum3,Mangontarum4}.
\end{itemize}

In 2012, Pan \cite{Pan} introduced a matrix decomposition that provides an explicit, nonrecursive computation of the values of the unified generalized Stirling numbers. More precisely, Pan \cite[Equation (8)]{Pan} obtained the following identity:
\begin{equation}
	\S_{\alpha,\beta,\gamma}=\S_{\alpha,0,0}\S_{0,0,\gamma}\S_{0,\beta,0},\label{Pan1}
\end{equation}
where $\S_{\alpha,\beta,\gamma}$ is the unified generalized Stirling (transform) matrix with real parameters $\alpha$, $\beta$, and $\gamma$, defined as an infinite-dimensional lower triangular matrix whose $(n,k)^{\rm th}$ entries are the unified generalized Stirling numbers $S(n,k;\alpha,\beta,\gamma)$, with $n,k=0,1,2,\ldots$. This matrix is shown to satisfy
\begin{equation}
	\v_\alpha(x)=\S_{\alpha,\beta,\gamma}\v_\beta(x-\gamma),
\end{equation}
where the notation $\v_\alpha(x)$ used in this relation denotes a column vector whose entries are the generalized factorials with real increment $\alpha$, given by
\begin{equation}
	\v_\alpha(x)=(1,x,(x|\alpha)_2,\ldots,(x|\alpha)_n,\ldots)^T.
\end{equation}
Notice that, in the matrix decomposition in \eqref{Pan1}, the $(n,k)^{\rm th}$ entries of the three matrices $\S_{\alpha,0,0}$, $\S_{0,\beta,0}$, and $\S_{0,0,\gamma}$ are $\alpha^{n-k}s(n,k)$, $\beta^{n-k}S(n,k)$, and $\gamma^{n-k}\binom{n}{k}$, respectively. 

A type of generalization of classical numbers is their $q$-analogues. A $q$-analogue generalizes a known mathematical expression and recovers the corresponding classical expression in the limit as $q$ tends to 1. Examples of $q$-analogues include the $q$-integer $n$, given by
$$ [n]_q=\frac{q^n-1}{q-1},$$
the $q$-factorial of $n$, given by
$$ [n]_q!=\prod_{j=1}^n[j]q,$$
the $q$-factorial of $n$ of order $k$, given by
$$ [n]_{k,q}=\prod_{j=0}^{k-1}\frac{q^{n-j}-1}{q-1},$$
and the $q$-binomial coefficient, given by
$$ \rstf{n}{k}_q
=\prod_{j=1}^k\frac{q^{n-j+1}-1}{q^j-1}
=\frac{[n]q!}{[k]q![n-k]q!},$$
since the following limits hold:
\begin{equation*}
	\lim_{q\rightarrow 1}[n]_q=n,\qquad
	\lim_{q\rightarrow 1}[n]_q!=n!,\qquad
	\lim_{q\rightarrow 1}[n]_{k,q}=(n)_k,\qquad
	\lim_{q\rightarrow 1}\rstf{n}{k}_q=\binom{n}{k}.
\end{equation*}

For the $q$-analogues of the unified generalized Stirling numbers, Corcino et al. \cite{Corcino4} defined the following:
\begin{equation}
	\sigma^1\big[n,k\big]\equiv\sigma^1\big[n,k;\alpha,\beta,\gamma\big]_q:=S\big[n,k;q^\alpha,q^\beta,q^\gamma-1\big](q-1)^{k-n}
	\label{corq1}
\end{equation}
and
\begin{equation}
	\sigma^2\big[n,k\big]\equiv\sigma^2\big[n,k;\alpha,\beta,\gamma\big]_q:=S\big[n,k;q^\beta,q^\alpha,1-q^\gamma\big](q-1)^{k-n},
\end{equation}
with $\sigma^1\big[0,0\big]=\sigma^2\big[0,0\big]=1$. Here, $S^1\big[n,k\big]\equiv S\big[n,k;q^\alpha,q^\beta,q^\gamma-1\big]$ and $S^2\big[n,k\big]\equiv S\big[n,k;q^\beta,q^\alpha,1-q^\gamma\big]$ are called exponential-type Stirling numbers and are defined by the pair of relations \cite[equations (3) and (4)]{Corcino4}
\begin{equation}
	[t|a]_n=\sum_{k=0}^{n}S^1[n,k][t-c|b]_k \label{etS1}
\end{equation}
and
\begin{equation}
	[t|b]_n=\sum_{k=0}^{n}S^2[n,k][t+c|a]_k,\label{etS2}
\end{equation}
where $a$, $b$, and $c$ are real or complex numbers, and
\begin{equation*}
	[t|a]_n=\prod_{j=0}^{n-1}(t-a^j),\qquad
	[t|a]_0=1,\qquad
	[t|a]_1=t-1.
\end{equation*}
For the sake of clarity, we refer to the numbers $\sigma^1\big[n,k\big]$ and $\sigma^2\big[n,k\big]$ as the unified generalized $q$-Stirling numbers. Many of the combinatorial properties of $\sigma^1\big[n,k\big]$ have already been developed by Corcino et al. \cite{Corcino3,Corcino4}, including the triangular recurrence relation
\begin{equation}\label{recqgenstir}
	\sigma^1\big[n+1,k\big]=\sigma^1\big[n,k-1\big]+([k\beta]_q-[n\alpha]_q-[\gamma]_q)\sigma^1\big[n,k\big],
\end{equation}
where $n\geq k\geq1$, $\sigma^1\big[n,n\big]=\sigma^2\big[0,0\big]$, and $\sigma^1\big[n,0\big]=[q^{\gamma}|q^{\alpha}]_n(q-1)^{-n}$, as well as the explicit formula
\begin{equation}
	\sigma^1\big[n,k\big]=\left(\prod_{i=0}^{k}[\beta i]_q\right)^{-1}\sum_{j=0}^{k}(-1)^{k-j}q^{\beta\left[\binom{j+1}{2}-kj\right]}\rstf{k}{j}_{q^{\beta}}\prod_{j=0}^{n-1}\left([\beta j]_q+[\gamma]_q-[\alpha i]_q\right).
\end{equation}
From these results, the following analogous properties for $S(n,k;\alpha,\beta,\gamma)$ are recovered in the limit as $q\rightarrow1$:
\begin{equation}
S(n+1,k;\alpha,\beta,\gamma)=S(n,k-1;\alpha,\beta,\gamma)+(k\beta-n\alpha+\gamma)S(n,k;\alpha,\beta,\gamma)\label{unirecur}
\end{equation}
\begin{equation}
	S(n,k;\alpha,\beta,\gamma)=\frac{1}{k!\beta^k}\sum_{j=0}^{k}(-1)^{k-j}\binom{k}{j}(\beta j+\gamma|\alpha)_n.
\end{equation}

In this paper, we define the unified generalized $q$-Stirling matrices and establish their matrix decompositions using a method similar to that employed by Pan \cite{Pan}.

\section{Unified Generalized $q$-Stirling Matrices}

By the definition of the factorial $\left[t|a\right]_q$, we have
\begin{align*}
	(q-1)^{-n}[q^t|q^{\alpha}]_q&=\prod_{j=0}^{n-1}\left(\frac{q^t-q^{\alpha j}}{q-1}\right)\\
	&=\prod_{j=0}^{n-1}\left(\frac{q^t-1}{q-1}-\frac{q^{\alpha j}-1}{q-1}\right)\\
	&=\prod_{j=0}^{n-1}\left([t]_q-[\alpha j]_q\right).
\end{align*}
Henceforth, we use $\left\langle \big[t\big]_q|\big[\alpha\big]_q\right\rangle_n$ to denote the expression on the right-hand side. That is,
\begin{equation}
	\left\langle \big[t\big]_q|\big[\alpha\big]_q\right\rangle_n=\prod_{j=0}^{n-1}\left([t]_q-[\alpha j]_q\right).
\end{equation}

Now, by letting $a=q^{\alpha}$, $b=q^{\beta}$, and $c=q^{\gamma}-1$, and multiplying both sides of \eqref{etS1} by $(q-1)^{-n}$, with $t$ replaced by $q^{t}$, we obtain
\begin{align*}
	\left[q^t|q^{\alpha}\right]_n&=\sum_{k=0}^{n}S[n,k;q^{\alpha},q^{\beta},q^{\gamma}-1]\left[q^t-q^{\gamma}+1|q^{\beta}\right]_k\\
	(q-1)^{-n}\left[q^t|q^{\alpha}\right]_n&=\sum_{k=0}^{n}S[n,k;q^{\alpha},q^{\beta},q^{\gamma}-1](q-1)^{k-n}\frac{\left[q^t-q^{\gamma}+1|q^{\beta}\right]_k}{(q-1)^k}.
\end{align*}
Since
\begin{equation}
	\frac{\left[q^t-q^{\gamma}+1|q^{\beta}\right]_k}{(q-1)^k}=\left\langle\big[t\big]_q-\big[\gamma\big]_q|\big[\beta\big]_q\right\rangle_k,
\end{equation}
it follows from \eqref{corq1} that
\begin{equation*}
	\left\langle\big[t\big]_q|\big[\alpha\big]_q\right\rangle_n=\sum_{k=0}^{n}\sigma^1\big[n,k\big]\left\langle\big[t\big]_q-\big[\gamma\big]_q|\big[\beta\big]_q\right\rangle_k.
\end{equation*}
Similarly, we may derive
\begin{equation*}
	\left\langle\big[t\big]_q|\big[\beta\big]_q\right\rangle_n=\sum_{k=0}^{n}\sigma^2\big[n,k\big]\left\langle\big[t\big]_q+\big[\gamma\big]_q|\big[\alpha\big]_q\right\rangle_k
\end{equation*}
from \eqref{etS2}. These two identities provide horizontal generating functions for the unified generalized $q$-Stirling numbers, which are essential for establishing the main results of this study. They also appear to be $q$-analogues of the defining relations in \eqref{Hsu1} and \eqref{Hsu2}. Hence, we formally state these results in the following theorem.

\begin{thm} \label{thm1}
	The unified generalized $q$-Stirling numbers $\sigma^1[n,k]$ and $\sigma^2[n,k]$ have the following horizontal generating functions:
	\begin{equation}
		\left\langle\big[t\big]_q|\big[\alpha\big]_q\right\rangle_n=\sum_{k=0}^{n}\sigma^1\big[n,k\big]\left\langle\big[t\big]_q-\big[\gamma\big]_q|\big[\beta\big]_q\right\rangle_k \label{result1}
	\end{equation} 
	and
	\begin{equation}
		\left\langle\big[t\big]_q|\big[\beta\big]_q\right\rangle_n=\sum_{k=0}^{n}\sigma^2\big[n,k\big]\left\langle\big[t\big]_q+\big[\gamma\big]_q|\big[\alpha\big]_q\right\rangle_k.\label{result2}
	\end{equation} 
\end{thm}

The horizontal generating functions in Theorem~\ref{thm1} naturally lead to a matrix representation of the unified generalized $q$-Stirling numbers. By arranging the numbers according to their indices, we obtain two infinite-dimensional lower triangular matrices, which we define as follows.

\begin{defn}\label{defnmajor}\rm 
	The unified generalized $q$-Stirling matrices with real parameters $\alpha$, $\beta$, and $\gamma$, denoted by $\sigma^1_{\alpha,\beta,\gamma}$ and $\sigma^2_{\alpha,\beta,\gamma}$, are defined to be infinite-dimensional lower triangular matrices, the $(n,k)^{\rm th}$ entries of which are $\sigma^1\big[n,k\big]$ and $\sigma^2\big[n,k\big]$, respectively.
\end{defn}

The two unified generalized $q$-Stirling matrices can be explicitly represented as
\begin{equation}
	\sigma^1_{\alpha,\beta,\gamma}=
	\begin{pmatrix}
		\sigma^1[0,0] & 0 & 0 & 0 & \cdots \\
		\sigma^1[1,0] & \sigma^1[1,1] & 0 & 0 & \cdots \\
		\sigma^1[2,0] & \sigma^1[2,1] & \sigma^1[2,2] & 0 & \cdots \\
		\sigma^1[3,0] & \sigma^1[3,1] & \sigma^1[3,2] & \sigma^1[3,3] & \cdots \\
		\vdots & \vdots & \vdots & \vdots & \ddots
	\end{pmatrix}
\end{equation}
and
\begin{equation}
	\sigma^2_{\alpha,\beta,\gamma}=
	\begin{pmatrix}
		\sigma^2[0,0] & 0 & 0 & 0 & \cdots \\
		\sigma^2[1,0] & \sigma^2[1,1] & 0 & 0 & \cdots \\
		\sigma^2[2,0] & \sigma^2[2,1] & \sigma^2[2,2] & 0 & \cdots \\
		\sigma^2[3,0] & \sigma^2[3,1] & \sigma^2[3,2] & \sigma^2[3,3] & \cdots \\
		\vdots & \vdots & \vdots & \vdots & \ddots
	\end{pmatrix}.
\end{equation}
Thus, the $(n,k)^{\rm th}$ entries of $\sigma^1_{\alpha,\beta,\gamma}$ and $\sigma^2_{\alpha,\beta,\gamma}$ are given by $\sigma^1[n,k]$ and $\sigma^2[n,k]$, respectively, for $n,k\geq 0$, with both matrices being lower triangular since $\sigma^1[n,k]=\sigma^2[n,k]=0$ whenever $k>n$.

The matrix representations introduced above provide a natural way to express the defining relations of the unified generalized $q$-Stirling numbers. In particular, the matrices $\sigma^1_{\alpha,\beta,\gamma}$ and $\sigma^2_{\alpha,\beta,\gamma}$ serve as transformation matrices between the generalized $q$-factorial sequences with parameters $\alpha$ and $\beta$. This leads to the following matrix relations.

\begin{thm}\label{thm2}
	The unified generalized $q$-Stirling matrices satisfy the following matrix relations:
	\begin{equation}
		\v_{\alpha}\left(\big[t\big]_q\right)=\sigma^1_{\alpha,\beta,\gamma}\v_{\beta}\left(\big[t\big]_q-\big[\gamma\big]_q\right)\label{result3}
	\end{equation} and
	\begin{equation}
		\v_{\beta}\left(\big[t\big]_q\right)=\sigma^2_{\alpha,\beta,\gamma}\v_{\alpha}\left(\big[t\big]_q+\big[\gamma\big]_q\right),\label{result4}
	\end{equation}
	where
	\begin{equation}
		\v_{\alpha}\left(\big[t\big]_q\right)=\left( 1,\big[t\big]_q,\left\langle\big[t\big]_q|\big[\alpha\big]_q\right\rangle_2,\left\langle\big[t\big]_q|\big[\alpha\big]_q\right\rangle_3,\ldots,\left\langle\big[t\big]_q|\big[\alpha\big]_q\right\rangle_n,\ldots\right)^T
	\end{equation}
\end{thm}

\begin{proof}
	For $n$ a nonnegative integer, let
	$\left\langle \big[t\big]_q|\big[\alpha\big]_q\right\rangle_n$
	denote the $n^{\rm th}$ entry of the column vector
	$\v_{\alpha}\left(\big[t\big]_q\right)$.
	By Equation \eqref{result1} in Theorem \ref{thm1},
	\begin{equation*}
		\left\langle \big[t\big]_q|\big[\alpha\big]_q\right\rangle_n=\sum_{k=0}^{n}\sigma^1\big[n,k\big]\left\langle\big[t\big]_q-\big[\gamma\big]_q\big|\big[\beta\big]_q\right\rangle_k.
	\end{equation*}
	Since $\sigma^1\big[n,k\big]=0$ for $k>n$, the right-hand side is precisely the inner product of the $n^{\rm th}$ row of the matrix $\sigma^1_{\alpha,\beta,\gamma}$ with the column vector
	$\v_{\beta}\left(\big[t\big]_q-\big[\gamma\big]_q\right)$.
	Therefore,
	\begin{equation*}
		\left\langle \big[t\big]_q|\big[\alpha\big]_q\right\rangle_n=\left[\sigma^1_{\alpha,\beta,\gamma}\v_{\beta}\left(\big[t\big]_q-\big[\gamma\big]_q\right)\right]_n,
	\end{equation*}
	where the subscript $n$ denotes the $n^{\rm th}$ entry of the resulting column vector.
	Since this holds for every $n\geq0$, it follows that
	\begin{equation}
		\v_{\alpha}\left(\big[t\big]_q\right)=\sigma^1_{\alpha,\beta,\gamma}\v_{\beta}\left(\big[t\big]_q-\big[\gamma\big]_q\right).
	\end{equation}
	
	Similarly, by Equation \eqref{result2} in Theorem \ref{thm1},
	\begin{equation*}
		\left\langle \big[t\big]_q|\big[\beta\big]_q\right\rangle_n=\sum_{k=0}^{n}
		\sigma^2\big[n,k\big]\left\langle\big[t\big]_q+\big[\gamma\big]_q\big|\big[\alpha\big]_q\right\rangle_k.
	\end{equation*}
	Since $\sigma^2\big[n,k\big]=0$ for $k>n$, the right-hand side is precisely the inner product of the $n^{\rm th}$ row of the matrix $\sigma^2_{\alpha,\beta,\gamma}$ with the column vector
	$\v_{\alpha}\left(\big[t\big]_q+\big[\gamma\big]_q\right)$.
	Hence,
	\begin{equation*}
		\left\langle \big[t\big]_q|\big[\beta\big]_q\right\rangle_n=\left[\sigma^2_{\alpha,\beta,\gamma}\v_{\alpha}\left(\big[t\big]_q+\big[\gamma\big]_q\right)\right]_n.
	\end{equation*}
	Again, since this holds for every $n\geq0$, we obtain
	\begin{equation}
		\v_{\beta}\left(\big[t\big]_q\right)=\sigma^2_{\alpha,\beta,\gamma}\v_{\alpha}\left(\big[t\big]_q+\big[\gamma\big]_q\right).
	\end{equation}
	This completes the proof.
\end{proof}

In 2016, Mangontarum et al. \cite{Mangontarum4} introduced a pair of $q$-analogues of the translated Whitney numbers defined by Belbachir and Bousbaa \cite{Bel}. Using the product
\begin{equation}
	\prod_{i=0}^{n-1}\big[x-i\alpha\big]_q=\big[x\big]_q\big[x-\alpha\big]_q\big[x-2\alpha\big]_q\cdots\big[x-(n-1)\alpha\big]_q,
\end{equation}
the translated $q$-Whitney numbers of the first and second kinds, denoted by
$w^1_{(\alpha)}\big[n,k\big]_q$ and
$w^2_{(\alpha)}\big[n,k\big]_q$, respectively, are defined by the following horizontal generating functions:
\begin{equation}
	\prod_{i=0}^{n-1}\big[x-i\alpha\big]_q=\sum_{k=0}^nw^1_{(\alpha)}\big[n,k\big]_q\big[x\big]_q^k,\label{transqfirst}
\end{equation}
\begin{equation}
	\big[x\big]_q^n=\sum_{k=0}^nw^2_{(\alpha)}\big[n,k\big]_q\prod_{i=0}^{k-1}\big[x-i\alpha\big]_q.\label{transqsecond}
\end{equation}
Several combinatorial properties of the translated $q$-Whitney numbers, which are $q$-analogues of corresponding properties of the classical translated Whitney numbers introduced by Belbachir and Bousbaa \cite{Bel} and presented by Mangontarum and Dibagulun \cite{Mangontarum3}, have already been established in the paper of Mangontarum et al. \cite{Mangontarum4}.

In the following corollaries, we use the translated $q$-Whitney numbers to describe the entries of the matrices
$\sigma^1_{\alpha,\beta,\gamma}$ and
$\sigma^2_{\alpha,\beta,\gamma}$
for particular values of $\alpha$, $\beta$, and $\gamma$.

\begin{cor}\label{thm3} 
	The $(n,k)^{\rm th}$ entry of the matrices $\sigma^1_{\alpha,0,0}$ and $\sigma^2_{\alpha,0,0}$ are given by
	\begin{equation}
	\sigma^1\big[n,k;\alpha,0,0\big]_q=q^{\alpha\binom{n}{2}}w_{(\alpha)}^{1}\big[n,k\big]_q\label{result5}
	\end{equation}
	and
	\begin{equation}
	\sigma^2\big[n,k;\alpha,0,0\big]_q=q^{-\alpha\binom{k}{2}}w_{(\alpha)}^{2}\big[n,k\big]_q,\label{result6}
	\end{equation}
	respectively.
\end{cor}

\begin{proof}
	To prove Equation \eqref{result5}, first observe that
	\begin{equation*}
		\big[t\big]_q-\big[\alpha j\big]_q=\frac{q^t-q^{\alpha j}}{q-1}=q^{\alpha j}\frac{q^{t-\alpha j}-1}{q-1}=q^{\alpha j}\big[t-\alpha j\big]_q.
	\end{equation*}
	Hence,
	\begin{align*}
		\left\langle \big[t\big]_q|\big[\alpha\big]_q\right\rangle_n
		&=\prod_{j=0}^{n-1}\left(\big[t\big]_q-\big[\alpha j\big]_q\right)\\
		&=\prod_{j=0}^{n-1}q^{\alpha j}\big[t-\alpha j\big]_q\\
		&=q^{\alpha\sum_{j=0}^{n-1}j}\prod_{j=0}^{n-1}\big[t-\alpha j\big]_q\\
		&=q^{\alpha\binom{n}{2}}\prod_{j=0}^{n-1}\big[t-\alpha j\big]_q.
	\end{align*}
	Now, setting $\beta=\gamma=0$ in Equation \eqref{result1} gives
	\begin{equation*}
		\left\langle \big[t\big]_q|\big[\alpha\big]_q\right\rangle_n=\sum_{k=0}^{n}\sigma^1\big[n,k;\alpha,0,0\big]_q\left\langle
		\big[t\big]_q-\big[0\big]_q\big|\big[0\big]_q\right\rangle_k.
	\end{equation*}
	Since $\big[0\big]_q=0$ and
	$\left\langle [t]_q|[0]_q\right\rangle_k=[t]_q^k$, it follows that
	\begin{equation*}
		\left\langle\big[t\big]_q|\big[\alpha\big]_q\right\rangle_n=\sum_{k=0}^{n}\sigma^1\big[n,k;\alpha,0,0\big]_q\big[t\big]_q^k.
	\end{equation*}
	Therefore,
	\begin{equation*}
		q^{\alpha\binom{n}{2}}\prod_{j=0}^{n-1}\big[t-\alpha j\big]_q=\sum_{k=0}^{n}\sigma^1\big[n,k;\alpha,0,0\big]_q
		\big[t\big]_q^k.
	\end{equation*}
	Using the defining relation for the translated $q$-Whitney numbers of the first kind in Equation \eqref{transqfirst}, we obtain
	\begin{equation*}
		\sum_{k=0}^{n}\sigma^1\big[n,k;\alpha,0,0\big]_q\big[t\big]_q^k=q^{\alpha\binom{n}{2}}\sum_{k=0}^{n}w^1_{(\alpha)}\big[n,k\big]_q\big[t\big]_q^k.
	\end{equation*}
	Comparing the coefficients of $\big[t\big]_q^k$ on both sides yields
	\begin{equation*}
		\sigma^1\big[n,k;\alpha,0,0\big]_q=q^{\alpha\binom{n}{2}}w^1_{(\alpha)}\big[n,k\big]_q,
	\end{equation*}
	which proves Equation \eqref{result5}.
	
	To prove Equation \eqref{result6}, setting $\beta=\gamma=0$ in Equation \eqref{result2} gives
	\begin{equation*}
		\left\langle\big[t\big]_q|\big[0\big]_q\right\rangle_n=\sum_{k=0}^{n}\sigma^2\big[n,k;\alpha,0,0\big]_q\left\langle\big[t\big]_q+\big[0\big]_q\big|\big[\alpha\big]_q\right\rangle_k.
	\end{equation*}
	Since
	$\left\langle [t]_q|[0]_q\right\rangle_n=[t]_q^n$
	and
	$\big[0\big]_q=0$, we have
	\begin{equation*}
		\big[t\big]_q^n=\sum_{k=0}^{n}\sigma^2\big[n,k;\alpha,0,0\big]_q\prod_{i=0}^{k-1}\left(\big[t\big]_q-\big[\alpha i\big]_q\right).
	\end{equation*}
	Using
	$$ \big[t\big]_q-\big[\alpha i\big]_q=q^{\alpha i}\big[t-\alpha i\big]_q,$$
	we obtain
	\begin{equation*}
		\big[t\big]_q^n=\sum_{k=0}^{n}q^{\alpha\binom{k}{2}}\sigma^2\big[n,k;\alpha,0,0\big]_q\prod_{i=0}^{k-1}\big[t-\alpha i\big]_q.
	\end{equation*}
	Comparing this with the defining relation for the translated $q$-Whitney numbers of the second kind in Equation \eqref{transqsecond}, we obtain
	\begin{equation*}
		q^{\alpha\binom{k}{2}}\sigma^2\big[n,k;\alpha,0,0\big]_q=w^2_{(\alpha)}\big[n,k\big]_q,
	\end{equation*}
	which proves Equation \eqref{result6}.
\end{proof}

\begin{cor}\label{thm4}
	The $(n,k)^{\rm th}$ entry of the matrices $\sigma^1_{0,\beta,0}$ and $\sigma^2_{0,\beta,0}$ are given by
	\begin{equation}
	\sigma^1\big[n,k;0,\beta,0\big]_q=q^{-\beta\binom{k}{2}}w_{(\beta)}^{2}\big[n,k\big]_q\label{result7}
	\end{equation}
	and
	\begin{equation}
	\sigma^2\big[n,k;0,\beta,0\big]_q=q^{\beta\binom{n}{2}}w_{(\beta)}^{1}\big[n,k\big]_q,\label{result8}
	\end{equation}
	respectively.
\end{cor}

\begin{proof}
	To prove Equation \eqref{result7}, set $\alpha=\gamma=0$ in Equation \eqref{result1}. We then obtain
	\begin{equation*}
		\left\langle \big[t\big]_q|\big[0\big]_q\right\rangle_n
		=
		\sum_{k=0}^{n}
		\sigma^1\big[n,k;0,\beta,0\big]_q
		\left\langle
		\big[t\big]_q-\big[0\big]_q
		\big|
		\big[\beta\big]_q
		\right\rangle_k.
	\end{equation*}
	Since $\big[0\big]_q=0$ and
	$\left\langle [t]_q|[0]_q\right\rangle_n=[t]_q^n$, this becomes
	\begin{equation*}
		\big[t\big]_q^n
		=
		\sum_{k=0}^{n}
		\sigma^1\big[n,k;0,\beta,0\big]_q
		\left\langle
		\big[t\big]_q
		\big|
		\big[\beta\big]_q
		\right\rangle_k.
	\end{equation*}
	Using the identity
	\begin{equation*}
		\big[t\big]_q-\big[\beta i\big]_q
		=
		q^{\beta i}\big[t-\beta i\big]_q,
	\end{equation*}
	we have
	\begin{align*}
		\left\langle
		\big[t\big]_q
		\big|
		\big[\beta\big]_q
		\right\rangle_k
		&=
		\prod_{i=0}^{k-1}
		\left(
		\big[t\big]_q-\big[\beta i\big]_q
		\right)\\
		&=
		q^{\beta\binom{k}{2}}
		\prod_{i=0}^{k-1}
		\big[t-\beta i\big]_q.
	\end{align*}
	Consequently,
	\begin{equation*}
		\big[t\big]_q^n
		=
		\sum_{k=0}^{n}
		q^{\beta\binom{k}{2}}
		\sigma^1\big[n,k;0,\beta,0\big]_q
		\prod_{i=0}^{k-1}
		\big[t-\beta i\big]_q.
	\end{equation*}
	On the other hand, by the defining relation for the translated $q$-Whitney numbers of the second kind in Equation \eqref{transqsecond},
	\begin{equation*}
		\big[t\big]_q^n
		=
		\sum_{k=0}^{n}
		w^2_{(\beta)}\big[n,k\big]_q
		\prod_{i=0}^{k-1}
		\big[t-\beta i\big]_q.
	\end{equation*}
	Comparing the coefficients of
	$\displaystyle\prod_{i=0}^{k-1}\big[t-\beta i\big]_q$
	on both sides gives
	\begin{equation*}
		q^{\beta\binom{k}{2}}
		\sigma^1\big[n,k;0,\beta,0\big]_q
		=
		w^2_{(\beta)}\big[n,k\big]_q,
	\end{equation*}
	which proves Equation \eqref{result7}.
	
	To prove Equation \eqref{result8}, set $\alpha=\gamma=0$ in Equation
	\eqref{result2}. We then obtain
	\begin{equation*}
		\left\langle \big[t\big]_q|\big[\beta\big]_q\right\rangle_n
		=
		\sum_{k=0}^{n}
		\sigma^2\big[n,k;0,\beta,0\big]_q
		\big[t\big]_q^k.
	\end{equation*}
	Using
	\begin{equation*}
		\left\langle \big[t\big]_q|\big[\beta\big]_q\right\rangle_n
		=
		q^{\beta\binom{n}{2}}
		\prod_{i=0}^{n-1}\big[t-i\beta\big]_q
	\end{equation*}
	and Equation \eqref{transqfirst}, we obtain
	\begin{equation*}
		q^{\beta\binom{n}{2}}
		\sum_{k=0}^{n}
		w^1_{(\beta)}\big[n,k\big]_q
		\big[t\big]_q^k
		=
		\sum_{k=0}^{n}
		\sigma^2\big[n,k;0,\beta,0\big]_q
		\big[t\big]_q^k.
	\end{equation*}
	Comparing the coefficients of $\big[t\big]_q^k$ gives
	\begin{equation*}
		\sigma^2\big[n,k;0,\beta,0\big]_q
		=
		q^{\beta\binom{n}{2}}
		w^1_{(\beta)}\big[n,k\big]_q,
	\end{equation*}
	which proves Equation \eqref{result8}.
\end{proof}

When $\alpha=\beta=0$, the generalized $q$-factorial reduces to
\[
\left\langle \big[t\big]_q|\big[0\big]_q\right\rangle_n
=
\prod_{j=0}^{n-1}
\left(\big[t\big]_q-\big[0j\big]_q\right)
=
\big[t\big]_q^n.
\]
Thus, by the binomial theorem,
\begin{align*}
	\big[t\big]_q^n
	&=
	\left(
	\big[t\big]_q-\big[\gamma\big]_q+\big[\gamma\big]_q
	\right)^n\\
	&=
	\sum_{k=0}^{n}
	\binom{n}{k}
	\big[\gamma\big]_q^{n-k}
	\left\langle
	\big[t\big]_q-\big[\gamma\big]_q
	\big|
	\big[0\big]_q
	\right\rangle_k.
\end{align*}
On the other hand, setting $\alpha=\beta=0$ in Equation \eqref{result1} gives
\begin{equation*}
	\big[t\big]_q^n
	=
	\sum_{k=0}^{n}
	\sigma^1\big[n,k;0,0,\gamma\big]_q
	\left\langle
	\big[t\big]_q-\big[\gamma\big]_q
	\big|
	\big[0\big]_q
	\right\rangle_k.
\end{equation*}
Comparing the coefficients of
$\left\langle
\big[t\big]_q-\big[\gamma\big]_q
\big|
\big[0\big]_q
\right\rangle_k$
therefore yields
\begin{equation*}
	\sigma^1\big[n,k;0,0,\gamma\big]_q
	=
	\big[\gamma\big]_q^{n-k}\binom{n}{k}.
\end{equation*}

Similarly, by writing
\begin{align*}
	\big[t\big]_q^n
	&=
	\left(
	\big[t\big]_q+\big[\gamma\big]_q-\big[\gamma\big]_q
	\right)^n\\
	&=
	\sum_{k=0}^{n}
	\binom{n}{k}
	\big(-\big[\gamma\big]_q\big)^{n-k}
	\left\langle
	\big[t\big]_q+\big[\gamma\big]_q
	\big|
	\big[0\big]_q
	\right\rangle_k,
\end{align*}
and setting $\alpha=\beta=0$ in Equation \eqref{result2}, we obtain
\begin{equation*}
	\big[t\big]_q^n
	=
	\sum_{k=0}^{n}
	\sigma^2\big[n,k;0,0,\gamma\big]_q
	\left\langle
	\big[t\big]_q+\big[\gamma\big]_q
	\big|
	\big[0\big]_q
	\right\rangle_k.
\end{equation*}
Comparing the coefficients of
$\left\langle
\big[t\big]_q+\big[\gamma\big]_q
\big|
\big[0\big]_q
\right\rangle_k$
gives
\begin{equation*}
	\sigma^2\big[n,k;0,0,\gamma\big]_q
	=
	\big(-\big[\gamma\big]_q\big)^{n-k}
	\binom{n}{k}.
\end{equation*}
These observations lead to the following corollary.

\begin{cor}\label{thm5}
	The $(n,k)^{\rm th}$ entries of the matrices
	$\sigma^1_{0,0,\gamma}$ and $\sigma^2_{0,0,\gamma}$
	are given by
	\begin{equation}
		\sigma^1\big[n,k;0,0,\gamma\big]_q
		=
		\big[\gamma\big]_q^{n-k}
		\binom{n}{k},
		\label{result9}
	\end{equation}
	and
	\begin{equation}
		\sigma^2\big[n,k;0,0,\gamma\big]_q
		=
		\big(-\big[\gamma\big]_q\big)^{n-k}
		\binom{n}{k},
		\label{result10}
	\end{equation}
	respectively.
\end{cor}

The matrix relations and special cases established in this section provide the necessary foundation for deriving the matrix decompositions of the unified generalized $q$-Stirling matrices presented in the following section.

\section{Matrix Decompositions of $\sigma^1_{\alpha,\beta,\gamma}$ and $\sigma^2_{\alpha,\beta,\gamma}$}

For the decomposition that follows, we first note that the vector
\[
\v_{\alpha}\left(\big[t\big]_q\right)
=
\left(
1,\big[t\big]_q,
\left\langle \big[t\big]_q|\big[\alpha\big]_q\right\rangle_2,
\left\langle \big[t\big]_q|\big[\alpha\big]_q\right\rangle_3,
\ldots
\right)^T
\]
consists of polynomials in the variable $\big[t\big]_q$, where
\[
\left\langle \big[t\big]_q|\big[\alpha\big]_q\right\rangle_n
=
\prod_{j=0}^{n-1}
\left(\big[t\big]_q-\big[\alpha j\big]_q\right),
\qquad n\geq 0.
\]
In particular, $\left\langle \big[t\big]_q|\big[\alpha\big]_q\right\rangle_n$
is a monic polynomial of degree $n$ in the variable $\big[t\big]_q$.
Thus, the sequence
\[
\left\{
\left\langle \big[t\big]_q|\big[\alpha\big]_q\right\rangle_n
\right\}_{n\geq0}
\]
forms a basis for the polynomial space in the variable $\big[t\big]_q$.
Consequently, a lower triangular matrix representing a transformation
between such polynomial bases is uniquely determined.

The following theorem establishes matrix decompositions of the unified generalized $q$-Stirling matrices $\sigma^1_{\alpha,\beta,\gamma}$ and $\sigma^2_{\alpha,\beta,\gamma}$.

\begin{thm}\label{thm6}
	Let $\sigma^1_{\alpha,\beta,\gamma}$ and $\sigma^2_{\alpha,\beta,\gamma}$ be the unified generalized $q$-Stirling matrices with real parameters $\alpha$, $\beta$, and $\gamma$. Then
	\begin{equation}
		\sigma^1_{\alpha,\beta,\gamma}=\sigma^1_{\alpha,0,0}\sigma^1_{0,0,\gamma}\sigma^1_{0,\beta,0},\label{finalresult1}
	\end{equation}
	and
	\begin{equation}
		\sigma^2_{\alpha,\beta,\gamma}=\sigma^2_{0,\beta,0}\sigma^2_{0,0,\gamma}\sigma^2_{\alpha,0,0}.\label{finalresult2}
	\end{equation}
\end{thm}

\begin{proof}
	By Theorem~\ref{thm2}, Equation~\eqref{result3} gives the following special cases:
	\begin{itemize}
		\item when $\alpha\neq 0$ and $\beta=\gamma=0$,
		\begin{equation}
			\v_{\alpha}\left(\big[t\big]_q\right)
			=
			\sigma_{\alpha,0,0}^1
			\v_{0}\left(\big[t\big]_q\right);
			\label{f1}
		\end{equation}
		\item when $\beta\neq 0$ and $\alpha=\gamma=0$,
		\begin{equation*}
			\v_{0}\left(\big[t\big]_q\right)
			=
			\sigma_{0,\beta,0}^1
			\v_{\beta}\left(\big[t\big]_q\right);
		\end{equation*}
		\item when $\alpha=\beta=0$ and $\gamma\neq 0$,
		\begin{equation*}
			\v_{0}\left(\big[t\big]_q\right)
			=
			\sigma_{0,0,\gamma}^1
			\v_{0}\left(\big[t\big]_q-\big[\gamma\big]_q\right).
		\end{equation*}
	\end{itemize}
	Combining these relations with Equation~\eqref{f1}, we obtain
	\begin{align*}
		\v_{\alpha}\left(\big[t\big]_q\right)
		&=
		\sigma_{\alpha,0,0}^1
		\v_{0}\left(\big[t\big]_q\right)\\
		&=
		\sigma_{\alpha,0,0}^1
		\sigma_{0,0,\gamma}^1
		\v_{0}\left(\big[t\big]_q-\big[\gamma\big]_q\right)\\
		&=
		\sigma_{\alpha,0,0}^1
		\sigma_{0,0,\gamma}^1
		\sigma_{0,\beta,0}^1
		\v_{\beta}\left(\big[t\big]_q-\big[\gamma\big]_q\right).
	\end{align*}
	On the other hand, Equation~\eqref{result3} gives
	\begin{equation*}
		\v_{\alpha}\left(\big[t\big]_q\right)
		=
		\sigma_{\alpha,\beta,\gamma}^1
		\v_{\beta}\left(\big[t\big]_q-\big[\gamma\big]_q\right).
	\end{equation*}
	Therefore,
	\begin{equation*}
		\sigma_{\alpha,\beta,\gamma}^1
		\v_{\beta}\left(\big[t\big]_q-\big[\gamma\big]_q\right)
		=
		\sigma_{\alpha,0,0}^1
		\sigma_{0,0,\gamma}^1
		\sigma_{0,\beta,0}^1
		\v_{\beta}\left(\big[t\big]_q-\big[\gamma\big]_q\right).
	\end{equation*}
	Since the entries of
	$\v_{\beta}\left(\big[t\big]_q-\big[\gamma\big]_q\right)$
	form a polynomial basis in the variable $\big[t\big]_q$, and all the matrices involved are lower triangular transformation matrices, the representing matrix is unique. Hence,
	\begin{equation*}
		\sigma_{\alpha,\beta,\gamma}^1
		=
		\sigma_{\alpha,0,0}^1
		\sigma_{0,0,\gamma}^1
		\sigma_{0,\beta,0}^1.
	\end{equation*}
	
	Similarly, by Theorem~\ref{thm2}, Equation~\eqref{result4} gives the following special cases:
	\begin{itemize}
		\item when $\beta\neq 0$ and $\alpha=\gamma=0$,
		\begin{equation}
			\v_{\beta}\left(\big[t\big]_q\right)
			=
			\sigma_{0,\beta,0}^2
			\v_{0}\left(\big[t\big]_q\right);
			\label{f2}
		\end{equation}
		\item when $\alpha\neq 0$ and $\beta=\gamma=0$,
		\begin{equation*}
			\v_{0}\left(\big[t\big]_q\right)
			=
			\sigma_{\alpha,0,0}^2
			\v_{\alpha}\left(\big[t\big]_q\right);
		\end{equation*}
		\item when $\alpha=\beta=0$ and $\gamma\neq 0$,
		\begin{equation*}
			\v_{0}\left(\big[t\big]_q\right)
			=
			\sigma_{0,0,\gamma}^2
			\v_{0}\left(\big[t\big]_q+\big[\gamma\big]_q\right).
		\end{equation*}
	\end{itemize}
	Using Equation~\eqref{f2}, we obtain
	\begin{align*}
		\v_{\beta}\left(\big[t\big]_q\right)
		&=
		\sigma_{0,\beta,0}^2
		\v_{0}\left(\big[t\big]_q\right)\\
		&=
		\sigma_{0,\beta,0}^2
		\sigma_{0,0,\gamma}^2
		\v_{0}\left(\big[t\big]_q+\big[\gamma\big]_q\right)\\
		&=
		\sigma_{0,\beta,0}^2
		\sigma_{0,0,\gamma}^2
		\sigma_{\alpha,0,0}^2
		\v_{\alpha}\left(\big[t\big]_q+\big[\gamma\big]_q\right).
	\end{align*}
	On the other hand, Equation~\eqref{result4} gives
	\begin{equation*}
		\v_{\beta}\left(\big[t\big]_q\right)
		=
		\sigma_{\alpha,\beta,\gamma}^2
		\v_{\alpha}\left(\big[t\big]_q+\big[\gamma\big]_q\right).
	\end{equation*}
	Consequently,
	\begin{equation*}
		\sigma_{\alpha,\beta,\gamma}^2
		\v_{\alpha}\left(\big[t\big]_q+\big[\gamma\big]_q\right)
		=
		\sigma_{0,\beta,0}^2
		\sigma_{0,0,\gamma}^2
		\sigma_{\alpha,0,0}^2
		\v_{\alpha}\left(\big[t\big]_q+\big[\gamma\big]_q\right).
	\end{equation*}
	Since the entries of
	$\v_{\alpha}\left(\big[t\big]_q+\big[\gamma\big]_q\right)$
	form a polynomial basis in the variable $\big[t\big]_q$, and all the matrices involved are lower triangular transformation matrices, the representing matrix is unique. Hence,
	\begin{equation*}
		\sigma_{\alpha,\beta,\gamma}^2
		=
		\sigma_{0,\beta,0}^2
		\sigma_{0,0,\gamma}^2
		\sigma_{\alpha,0,0}^2.
	\end{equation*}
	This completes the proof.
\end{proof}

The matrix decompositions established in Theorem~\ref{thm6} provide a compact, explicit, and nonrecursive representation of the unified generalized $q$-Stirling numbers. As shown in Equation~\eqref{finalresult1}, the unified generalized $q$-Stirling matrix $\sigma^1_{\alpha,\beta,\gamma}$ is expressed as the product of the matrices $\sigma^1_{\alpha,0,0}$, $\sigma^1_{0,0,\gamma}$, and $\sigma^1_{0,\beta,0}$. Similarly, Equation~\eqref{finalresult2} expresses the unified generalized $q$-Stirling matrix $\sigma^2_{\alpha,\beta,\gamma}$ as the product of the matrices $\sigma^2_{0,\beta,0}$, $\sigma^2_{0,0,\gamma}$, and $\sigma^2_{\alpha,0,0}$.

\section{Conclusion}

In this paper, we introduced the unified generalized $q$-Stirling matrices associated with the unified generalized $q$-Stirling numbers and established their fundamental matrix relations. Using these relations and the corresponding special cases, we derived matrix decompositions of $\sigma^1_{\alpha,\beta,\gamma}$ and $\sigma^2_{\alpha,\beta,\gamma}$ into products of matrices corresponding to simpler parameter configurations. These results extend the matrix decomposition approach of Pan \cite{Pan} to the unified generalized $q$-Stirling setting and provide a convenient matrix framework for further investigations of their algebraic and combinatorial properties.

The matrix method introduced by Pan \cite{Pan} has also been extended to other settings. Bent-Usman et al. \cite{Wardah} extended this approach to the $q$-analogue setting in their study of the $q$-Ruci\'nski--Voigt numbers. More recently, Kabirun and Montero \cite{Kabirun} investigated the matrix decomposition of the $r$-Whitney numbers of both kinds and established properties of the associated generalized factorial matrices, particularly the triangular matrix factors of their inverses. The decomposition formulas obtained in their work bear similarities to those previously reported by Mangontarum et al. \cite{Mangontarum5}. These related results provide additional context for the matrix decompositions established in the present study.

\end{document}